\documentclass{amsart}
\usepackage{amsmath,amssymb,graphicx,latexsym}
\usepackage{hyperref}

\newtheorem{theorem}{Theorem}
\newtheorem{corollary}[theorem]{Corollary}
\newtheorem{proposition}[theorem]{Proposition}

\theoremstyle{definition}
\newtheorem{example}[theorem]{Example}

\newcommand{\bfs}{\mathbf{s}}

\newcommand{\C}{\mathbb{C}}
\newcommand{\N}{\mathbb{N}}
\newcommand{\Q}{\mathbb{Q}}
\newcommand{\R}{\mathbb{R}}
\newcommand{\Z}{\mathbb{Z}}

\newcommand{\nequiv}{\mathrel{\not\equiv}}
\newcommand{\colonequal}{\mathrel{\mathop:}=}

\newcommand{\size}[1]{\lvert{#1}\rvert}

\begin{document}

\title[Lucas congruences for the Ap\'ery numbers modulo $p^2$]{
Lucas congruences \\ for the Ap\'ery numbers modulo $p^2$
}

\author{Eric Rowland}
\address{
	Department of Mathematics \\
	Hofstra University \\
	Hempstead, NY \\
	USA
}
\author{Reem Yassawi}
\address{
School of Mathematics and Statistics\\
Open University\\
Milton Keynes \\
UK
}
\author{Christian Krattenthaler}
\address{
	Fakult\"at f\"ur Mathematik \\
	Universit\"at Wien \\
	Vienna \\
	Austria
}

\date{January 25, 2021}

\begin{abstract}
The sequence $A(n)_{n \geq 0}$ of Ap\'ery numbers can be interpolated to $\C$ by an entire function.
We give a formula for the Taylor coefficients of this function, centered at the origin, as a $\Z$-linear combination of multiple zeta values.
We then show that for integers $n$ whose base-$p$ digits belong to a certain set, $A(n)$ satisfies a Lucas congruence modulo $p^2$.
\end{abstract}

\maketitle

\section{Introduction}\label{introduction}

\makeatletter
\renewcommand*{\@makefnmark}{}
\footnotetext{This paper was originally posted on the arXiv by the first two authors.
Christian Krattenthaler became a coauthor after improving the proof of Theorem~\ref{multiple zeta value formula}.}
\makeatother

For each integer $n \geq 0$, the $n$th \emph{Ap\'ery number} is defined by
\[
	A(n) \colonequal \sum_{k \geq 0} \binom{n}{k}^2 \binom{n + k}{k}^2.
\]
These numbers arose in Ap\'ery's proof of the irrationality of $\zeta(3)$.
This sum is finite, since $\binom{n}{k} = 0$ when $k > n$.
The sequence $A(n)_{n \geq 0}$ is
\[
	1, 5, 73, 1445, 33001, 819005, 21460825, 584307365, \dots.
\]
The Ap\'ery numbers satisfy the recurrence
\begin{equation}\label{Apery recurrence}
	n^3 A(n) - (34 n^3 - 51 n^2 + 27 n - 5) A(n - 1) + (n - 1)^3 A(n - 2) = 0
\end{equation}
for all integers $n \geq 2$.

Exceptional properties of the Ap\'{e}ry sequence have been observed in many settings~\cite{Zagier 2017}.
Gessel~\cite{Gessel} showed that the Ap\'ery numbers satisfy the Lucas congruence
\begin{equation}\label{Lucas congruence}
	A(d + p n) \equiv A(d) A(n) \mod p
\end{equation}
for all $d \in \{0, 1, \dots, p - 1\}$ and $n \geq 0$.
Beukers~\cite{Beukers} established the supercongruence $A(p^\alpha n - 1) \equiv A(p^{\alpha - 1} n - 1) \mod p^{3 \alpha}$ for all primes $p \geq 5$, and Straub~\cite{Straub} showed that a related supercongruence holds more generally for a four-dimensional sequence containing $A(n)_{n \geq 0}$ as its diagonal.

Gessel also extended Congruence~\eqref{Lucas congruence} to a congruence modulo $p^2$ as follows.
Define the sequence $A'(n)_{n \geq 0}$ by
\begin{equation}\label{Gessel Apery derivative}
	A'(n) \colonequal 2 \sum_{k = 0}^n \binom{n}{k}^2 \binom{n + k}{k}^2 \left(H_{n + k} - H_{n - k}\right),
\end{equation}
where $H_k = 1 + \frac{1}{2} + \dots + \frac{1}{k}$ is the $k$th harmonic number.
The sequence $A'(n)_{n \geq 0}$ is
\[
	0, 12, 210, 4438, 104825, \frac{13276637}{5}, 70543291, \frac{67890874657}{35}, \dots.
\]
Then
\begin{equation}\label{Gessel congruence modulo p^2}
	A(d + p n) \equiv \left(A(d) + p n A'(d)\right) A(n) \mod p^2
\end{equation}
for all $d \in \{0, 1, \dots, p - 1\}$ and for all $n \geq 0$~\cite[Theorem~4]{Gessel}.

Gessel remarks that if $A(n)$ can be extended to a differentiable function $A(x)$ defined for $x \in \R_{\geq 0}$ such that $A(x)$ satisfies Recurrence~\eqref{Apery recurrence}, then $A'(n) = \left(\frac{d}{dx} A(x)\right)\vert_{x=n}$.
As shown by Zagier~\cite[Proposition 1]{Zagier 2017} and proved in an automated way by Osburn and Straub~\cite[Remark 2.5]{Osburn--Straub}, $A(n)$ can be extended to an entire function $A(z)$ satisfying
\begin{multline}\label{interpolation}
	z^3 A(z) - (34 z^3 - 51 z^2 + 27 z - 5) A(z - 1) + (z - 1)^3 A(z - 2) \\
	= \frac{8}{\pi^2} (2 z - 1) (\sin(\pi z))^2
\end{multline}
for all $z \in \C$.
Since both $\frac{8}{\pi^2} (2 z - 1) (\sin(\pi z))^2$ and its derivative vanish at integer values of $z$, it follows that $A'(n) = \left(\frac{d}{dz} A(z)\right)\vert_{z=n}$, hence the notation $A'(n)$.
Therefore the extension $A(z)$ confirms Gessel's intuition.

In this article we use an elementary approach to write the coefficients in the Taylor series of $A(z) = \sum_{m \geq 0} a_m z^m$ at $z = 0$ as an explicit $\Z$-linear combination of multiple zeta values.
A striking fact is that the coefficient of each multiple zeta value is a signed power of~$2$.
Let $s_1, s_2, \dots, s_j$ be positive integers with $s_1 \geq 2$.
The \emph{multiple zeta value} $\zeta(s_1, s_2, \dots, s_j)$ is defined as
\[
	\zeta(s_1, s_2, \dots, s_j)
	\colonequal \sum_{n_1 > n_2 > \dots > n_j > 0} \frac{1}{n_1^{s_1} n_2^{s_2} \cdots n_j^{s_j}}.
\]
The \emph{weight} of $\zeta(s_1, s_2, \dots, s_j)$ is $s_1 + s_2 + \dots + s_j$.

Let $\chi(m)$ be the characteristic function of the set of odd numbers.
That is, $\chi(m) = 0$ if $m$ is even and $\chi(m) = 1$ if $m$ is odd.
For a tuple $\bfs = (s_1, s_2, \dots, s_j)$, let $e(\bfs) = \size{\{i : \text{$2 \leq i \leq j$ and $s_i = 2$}\}}$.

\begin{theorem}\label{multiple zeta value formula}
Let $A(z) = \sum_{m \geq 0} a_m z^m$ be the Taylor series of the Ap\'ery function, centered at the origin.
For each $m \geq 1$,
\[
	a_m =
	\sum_{\bfs} (-1)^{\frac{m - s_1}{2}} 2^{e(\bfs) + \chi(m)} \zeta(s_1, s_2, \dots, s_j),
\]
where the sum is over all tuples $\bfs = (s_1, s_2, \dots, s_j)$, with $j \geq 1$, of non-negative integers satisfying
\begin{itemize}
\item $s_1 + s_2 + \dots + s_j = m$,
\item $s_1 = 3$ if $m$ is odd and $s_1 \in \{2, 4\}$ if $m$ is even, and
\item $s_i \in \{2, 4\}$ for all $i \in \{2, \dots, j\}$.
\end{itemize}
\end{theorem}

The first several coefficients are
\begin{align*}
	a_0 &= 1 \\
	a_1 &= 0 \\
	a_2 &= \zeta(2) \\
	a_3 &= 2 \zeta(3) \\
	a_4 &= \zeta(4) - 2 \zeta(2,2) \\
	a_5 &= -4 \zeta(3, 2) \\
	a_6 &= \zeta(2, 4) - 2 \zeta(4, 2) + 4 \zeta(2, 2, 2) \\
	a_7 &= 2 \zeta(3, 4) + 8 \zeta(3, 2, 2) \\
	a_8 &= \zeta(4, 4) - 2 \zeta(2, 2, 4) - 2 \zeta(2, 4, 2) + 4 \zeta(4, 2, 2) - 8 \zeta(2, 2, 2, 2) \\
	a_9 &= -4 \zeta(3, 2, 4) - 4 \zeta(3, 4, 2) - 16 \zeta(3, 2, 2, 2).
\end{align*}
Let $F(m)$ be the $m$th Fibonacci number.
Since the number of integer compositions of $m$ using parts $1$ and $2$ is $F(m + 1)$, Theorem~\ref{multiple zeta value formula} expresses $a_m$ as a linear combination of $F(\frac{m}{2} + 1)$ multiple zeta values if $m$ is even and $F(\frac{m - 1}{2})$ multiple zeta values if $m$ is odd.

Let $P(m)$ be the number of integer compositions of $m - 3$ using parts $2$ and~$3$.
Then $P(m)$ is the $m$th Padovan number and satisfies the recurrence $P(m) = P(m - 2) + P(m - 3)$ with initial conditions $P(3) = 1$, $P(4) = 0$, $P(5) = 1$.
Let $d_m$ be the dimension of the $\Q$-vector space spanned by the weight-$m$ multiple zeta values.
Recent progress by Brown~\cite{Brown} shows that $d_m \leq P(m + 3)$.
For $m \geq 13$, the representation of $a_m$ in Theorem~\ref{multiple zeta value formula} uses fewer than $P(m + 3)$ multiple zeta values.
Since $F(\frac{m}{2} + 1) > P(m + 3)$ for $m \in \{4, 6, 8, 10, 12\}$, this implies that $a_4, a_6, a_8, a_{10}, a_{12}$ can be written as $\Q$-linear combinations of fewer multiple zeta values than Theorem~\ref{multiple zeta value formula} provides.
Namely,
\begin{align*}
	a_4 &= -\tfrac{1}{2} \zeta(4) \\
	a_6 &= \tfrac{3}{2} \zeta(6) - 3 \zeta(4,2) \\
	a_8 &= -\tfrac{13}{24} \zeta(8) + 6 \zeta(4,2,2) \\
	a_{10} &= \tfrac{7}{8} \zeta(10) + 3 \zeta(2, 4, 4) - 12 \zeta(4, 2, 2, 2) \\
	a_{12} &= -\tfrac{915}{22112} \zeta(12) + 6 \zeta(4, 2, 2, 4) + 6 \zeta(4, 2, 4, 2) + 6 \zeta(4, 4, 2, 2) + 24 \zeta(4, 2, 2, 2, 2).
\end{align*}
We prove Theorem~\ref{multiple zeta value formula} in Section~\ref{Taylor coefficients}.
The proof technique can also be applied to compute the Taylor coefficients for a larger family of hypergeometric functions.
We remark that there are some parallels between Theorem~\ref{multiple zeta value formula} and work of Cresson, Fischler, and Rivoal~\cite{Cresson--Fischler--Rivoal}, who show that a class of hypergeometric series can be decomposed as $\Q$-linear combinations of multiple zeta values.
Numerically, Golyshev and Zagier~\cite[Section~2.4]{Golyshev--Zagier} also obtained multiple zeta values in coefficients of a formal power series related to the Ap\'ery numbers.

Returning to congruences for $A(n)$ in Section~\ref{congruences}, we consider the following question.
For which base-$p$ digits $d$ does Congruence~\eqref{Lucas congruence} hold not just modulo $p$ but modulo $p^2$?
The following theorem characterizes such digits.
Let
\[
	D(p) = \left\{d \in \{0, 1, \dots, p - 1\} : A(d) \equiv A(p - 1 - d) \mod p^2\right\}.
\]

\begin{theorem}\label{Lucas congruence p^2}
Let $p$ be a prime, and let $d \in \{0, 1, \dots, p - 1\}$.
The congruence $A(d + p n) \equiv A(d) A(n) \mod p^2$ holds for all $n \in \Z$ if and only if $d \in D(p)$.
\end{theorem}

In particular, if $n$ is a non-negative integer and all digits in its standard base-$p$ representation $n_\ell \cdots n_1 n_0$ belong to $D(p)$, then
\[
	A(n) \equiv A(n_0) A(n_1) \cdots A(n_\ell) \mod p^2.
\]
Theorem~\ref{Lucas congruence p^2} has an analogue for binomial coefficients, established by the first-named author~\cite{Rowland}.

\section{Taylor coefficients of the Ap\'ery function}\label{Taylor coefficients}

In this section we give a proof of Theorem~\ref{multiple zeta value formula}.
Let $\N = \{0, 1, 2, \dots\}$.
The sequence $A(n)_{n \geq 0}$ can be interpolated to $\C$ using the gamma function $\Gamma(z)$.
Recall that $\Gamma(z)$ is a meromorphic function satisfying
\[
	\Gamma(1)=1 \mbox{ and } \Gamma(z+1)=z\Gamma(z)
\]
for $z\not\in -\N$.
The gamma function has simple poles at the non-positive integers.

For $n \geq 0$, we can write $A(n)$ as
\begin{align*}
	A(n)
	&= \sum_{k \geq 0} \binom{n}{k}^2 \binom{n + k}{k}^2 \\
	&= \sum_{k \geq 0} \frac{\Gamma(n + k + 1)^2}{\Gamma(n - k + 1)^2 \Gamma(k + 1)^4}.
\end{align*}
We extend $A(n)$ to complex values by defining
\[
	A(z) = \sum_{k \geq 0} \frac{\Gamma(z + k + 1)^2}{\Gamma(z - k + 1)^2 \Gamma(k + 1)^4}.
\]
Note that for each $k\in\N$ the function $\frac{\Gamma(z + k + 1)^2}{\Gamma(z - k + 1)^2 \Gamma(k+1)^4}$ is a polynomial in $z$.
Furthermore, for each $z\in\C$, the series $\sum_{k\geq 0} \frac{\Gamma(z + k + 1)^2}{\Gamma(z - k + 1)^2 \Gamma(k+1)^4}$ is locally uniformly convergent.
Thus $A(z)$ is an entire function, which we call the {\em Ap\'{e}ry function}.
We remark that $A(z)$ can be written using the hypergeometric function ${}_4F_3$.
Let $(z)_k \colonequal z (z + 1) (z + 2) \cdots (z + k - 1)$ be the Pochhammer symbol (rising factorial).
By writing $\frac{\Gamma(z + k + 1)^2}{\Gamma(z - k + 1)^2} = (-z)_k^2 (z + 1)_k^2$, we see that
\begin{align}
	A(z)
	&= \sum_{k \geq 0} \frac{(-z)_k (-z)_k (z + 1)_k (z + 1)_k}{k!^4} \label{Apery hypergeometric function} \\
	&= {}_4F_3(-z, -z, z + 1, z + 1; 1, 1, 1; 1). \nonumber
\end{align}

Straub~\cite[Remark~1.3]{Straub} proved the reflection formula $A(-1 - n) = A(n)$ for all $n \in \Z$.
Equation~\eqref{Apery hypergeometric function} shows that this formula also holds for non-integers, since the hypergeometric series is invariant under replacing $z$ with $-1 - z$.

\begin{proposition}\label{reflection formula}
For all $z \in \C$, we have $A(-1 - z) = A(z)$.
\end{proposition}

\begin{figure}
	\includegraphics[width=.7\textwidth]{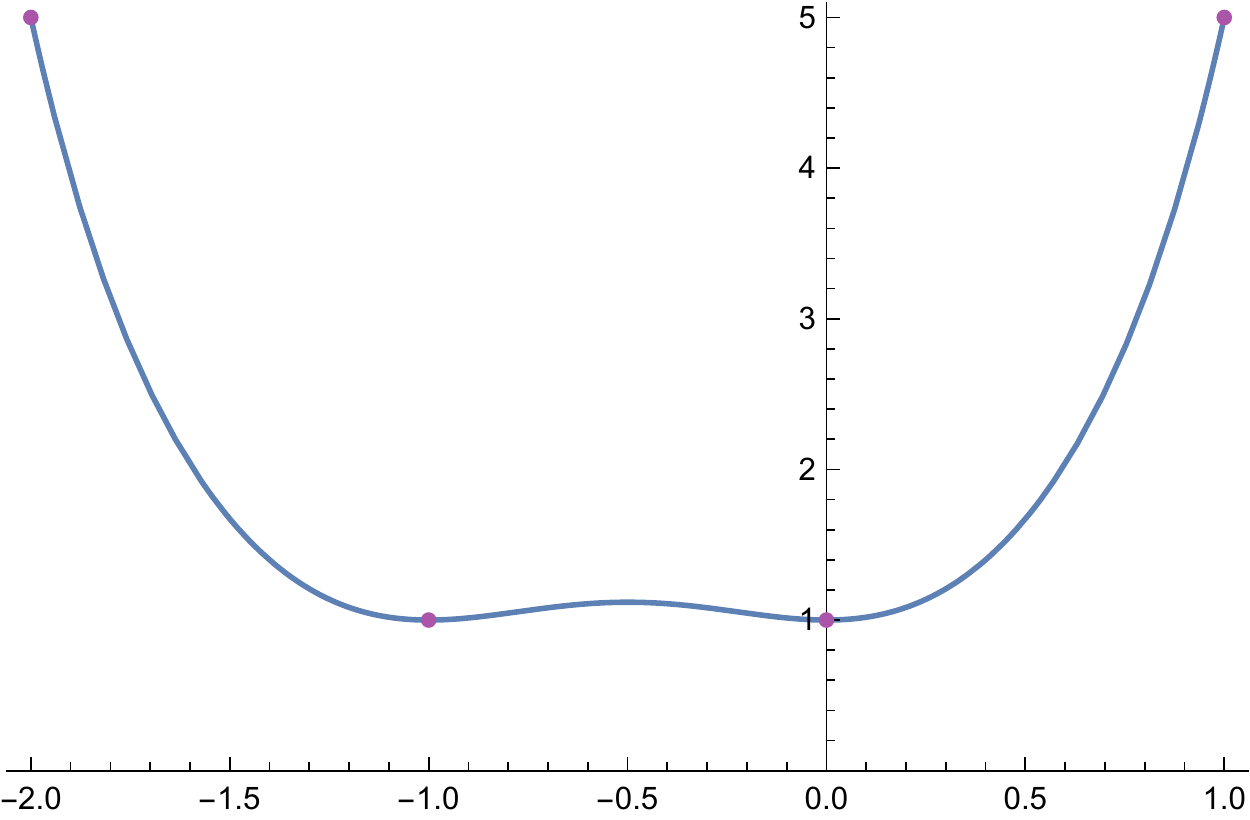}
	\caption{A plot of $A(z)$ for real $z$ in the interval $-2 \leq z \leq 1$, showing the reflection symmetry $A(-1 - z) = A(z)$.}
	\label{plot}
\end{figure}

Figure~\ref{plot} shows this symmetry on the real line.
In light of Proposition~\ref{reflection formula}, Theorem~\ref{multiple zeta value formula} also gives us the Taylor expansion of $A(z)$ at $z = -1$ for free.
We note that, at the symmetry point $z=-\frac{1}{2}$, Zagier has shown that $A(-\frac{1}{2})=\frac{16}{\pi^2}L(f,2)$ where $L(f,2)$ is the critical $L$-value of $f$, the unique normalized Hecke eigenform of weight $4$ for $\Gamma_0(8)$; see \cite{Zagier 2017} for an account and \cite{Zudilin} for a generalization.
There is no reason to expect that the Taylor coefficients of $A(z)$ centered at non-integer points are $\Q$-linear combinations of multiple zeta values.

Let
\begin{equation}\label{series expansion}
	A(z) = \sum_{k \geq 0}\frac{\Gamma(z + k + 1)^2}{\Gamma(z - k + 1)^2 \Gamma(k+1)^4} = \sum_{m \geq 0} a_m z^m
\end{equation}
be the Taylor series expansion of the Ap\'{e}ry function centered at the origin.
It is possible to compute $a_m$ by directly evaluating the $m$th derivative $A^{(m)}(z)$ at $z = 0$.

\begin{example}
The derivative of the summand is
\[
	\tfrac{1}{k!^4} \tfrac{d}{d z} \tfrac{\Gamma(z + k + 1)^2}{\Gamma(z - k + 1)^2}
	= \tfrac{1}{k!^4} \tfrac{\Gamma(z + k + 1)^2}{\Gamma(z - k + 1)^2}
	\left(2 \psi(z + k + 1) - 2 \psi(z - k + 1)\right),
\]
where the digamma function $\psi(z) \colonequal \frac{\Gamma'(z)}{\Gamma(z)}$ is the logarithmic derivative of $\Gamma(z)$.
This agrees with the expression for $A'(n)$ in Equation~\eqref{Gessel Apery derivative}.
Since $\frac{\Gamma(z + k + 1)^2}{\Gamma(z - k + 1)^2} = O(z^2)$ as $z \to 0$ and $2 \psi(z + k + 1) - 2 \psi(z - k + 1)$ has a simple pole at $0$ for each $k$, we have $a_1 = \frac{A'(0)}{1!} = 0$.
Similarly, the second derivative is
\begin{multline*}
	\tfrac{1}{k!^4} \tfrac{d^2}{d z^2} \tfrac{\Gamma(z + k + 1)^2}{\Gamma(z - k + 1)^2}
	= \tfrac{1}{k!^4} \tfrac{\Gamma(z + k + 1)^2}{\Gamma(z - k + 1)^2}
	\big(4 \psi(z + k + 1)^2 + 2 \psi'(z + k + 1) \\
	- 8 \psi(z + k + 1) \psi(z - k + 1) + 4 \psi(z - k + 1)^2 - 2 \psi'(z - k + 1)\big).
\end{multline*}
The series expansions of $\psi(z + k + 1)$ and $\psi(z - k + 1)$ imply $A''(0) = \sum_{k \geq 1} \frac{ 2}{ k^2}= 2\zeta(2)$, so $a_2 = \frac{A''(0)}{2!} = \zeta(2)$.
\end{example}

Theorem~\ref{multiple zeta value formula} can be proved by carrying out the same approach for general $m$.
However, we give a shorter proof in the spirit of \cite[Section~1.4]{Fischler}.

\begin{proof}[Proof of Theorem~\ref{multiple zeta value formula}]
We consider the summand in Equation~\eqref{series expansion}.
For $k=0$, we have $\frac {\Gamma^2(z+k+1)} {\Gamma^2(z-k+1) k!^4}=1$. 
For $k\geq 1$, we have
\begin{align}\label{summand}
\nonumber\frac {\Gamma^2(z+k+1)} {\Gamma^2(z-k+1) k!^4}
&=\tfrac {(z-k+1)^2\cdots(z-1)^2z^2(z+1)^2\cdots(z+k)^2} {k!^4}\\
\nonumber&=
\left(1-\tfrac {z} {k-1}\right)^2
\cdots
\left(1-\tfrac {z} {1}\right)^2
\left(1+\tfrac {z} {1}\right)^2
\cdots
\left(1+\tfrac {z} {k-1}\right)^2
\tfrac {z^2} {k^2}\left(1+\tfrac {z} {k}\right)^2\\
\nonumber&=
\left(1-\tfrac {z^2} {(k-1)^2}\right)^2
\cdots
\left(1-\tfrac {z^2} {1^2}\right)^2
\tfrac {z^2} {k^2}\left(1+\tfrac {z} {k}\right)^2\\
&=
\left(1-2\tfrac {z^2} {1^2}+\tfrac {z^4} {1^4}\right)
\cdots
\left(1-2\tfrac {z^2} {(k-1)^2}+\tfrac {z^4} {(k-1)^4}\right)
\left(\tfrac {z^2} {k^2}+2\tfrac {z^3} {k^3}+\tfrac {z^4} {k^4}\right).
\end{align}
Recall that $\chi(m)$ is the characteristic function of the set of odd numbers, and $e(\bfs) = \size{\{i : \text{$2 \leq i \leq j$ and $s_i = 2$}\}}$ for a tuple $\bfs = (s_1, s_2, \dots, s_j)$.
By expanding
the product \eqref{summand} to extract the coefficient of $z^m$, one sees that this coefficient equals
$$
\sum_{\substack{\mathbf s=(s_1,\dots,s_j) \\ s_1+\dots+s_j=m}}
\sum_{k=n_1>n_2>\dots>n_j>0}
(-1)^{\frac{m - s_1}{2}} 2^{e(\bfs) + \chi(m)}
\frac {1} {n_1^{s_1} n_2^{s_2}\cdots n_j^{s_j}},
$$
where the outer sum is over all $\bfs$ described in the statement of Theorem~\ref{multiple zeta value formula}.
Now we sum over all $k$ to obtain $a_m$, and the statement follows.
\end{proof}

As discussed in Section~\ref{introduction}, the coefficients $a_4, a_6, a_8, a_{10}, a_{12}$ can be written as $\Q$-linear combinations of fewer multiple zeta values than given by Theorem~\ref{multiple zeta value formula}.
The strategy given in the following example can be used to reduce $a_m$ for all even $m \geq 4$.

\begin{example}
For $m = 10$, Theorem~\ref{multiple zeta value formula} gives
\begin{multline*}
	a_{10} = \zeta(2, 4, 4) - 2 \zeta(4, 2, 4) - 2 \zeta(4, 4, 2) \\
	+ 4 \zeta(2, 2, 2, 4) + 4 \zeta(2, 2, 4, 2) + 4 \zeta(2, 4, 2, 2) - 8 \zeta(4, 2, 2, 2) \\
	+ 16 \zeta(2, 2, 2, 2, 2).
\end{multline*}
We will rewrite several products $\zeta(s_1, s_2, \dots, s_j) \zeta(i)$ as linear combinations of multiple zeta values.
For example,
\begin{multline*}
	\left(\sum_{k_1>k_2>0 } \frac{1}{k_{1}^{a} k_{2}^{b}}\right) \left(\sum_{k_3 > 0} \frac{1}{k_{3}^{c}}\right) \\
	= \sum_{k_3>k_1>k_2>0 } \frac{1}{k_{1}^{a}k_{2}^{b}k_{3}^{c}}
	+ \sum_{k_1>k_3>k_2>0 } \frac{1}{k_{1}^{a}k_{2}^{b} k_{3}^{c}}
	+ \sum_{k_1>k_2>k_3 >0 } \frac{1}{k_{1}^{a} k_{2}^{b}k_{3}^{c}} \\
	+ \sum_{k_1>k_2>0 } \frac{1}{k_{1}^{a+c}k_{2}^{b}} 
	+ \sum_{k_1>k_2>0 } \frac{1}{k_{1}^{a}k_{2}^{b+c}}, 
\end{multline*}
so that 
\begin{align}\label{zeta relation}
	\zeta(a,b)\zeta(c) = \zeta(c,a,b) + \zeta(a,c,b) + \zeta(a,b,c) + \zeta(a+c,b)+ \zeta(a,b+c).
\end{align}
As in the derivation of Equation~\eqref{zeta relation}, we have $\zeta(a)\zeta(b)= \zeta(a,b)+\zeta(b,a)+\zeta(a+b)$.

We first express $-2 \zeta(4, 2, 4) - 2 \zeta(4, 4, 2)$ in terms of $\zeta(2, 4, 4)$ and $\zeta(10)$.
By \eqref{zeta relation} we have
\[
	\zeta(4, 4) \zeta(2)
	= \zeta(2,4,4) + \zeta(4,2,4) + \zeta(4,4,2) + \zeta(6,4) + \zeta(4,6).
\]
The relations $\zeta(4) \zeta(4) = 2 \zeta(4, 4) + \zeta(8)$ and $\zeta(4) \zeta(6) = \zeta(4, 6) + \zeta(6, 4) + \zeta(10)$ allow us to write
\begin{align*}
	-2 \zeta(4, 2, 4) - 2 \zeta(4, 4, 2)
	&= 2 \zeta(2, 4, 4) + 2 \zeta(4) \zeta(6) - 2 \zeta(10) - \zeta(4)^2 \zeta(2) + \zeta(8) \zeta(2) \\
	&= 2 \zeta(2, 4, 4) - \tfrac{3}{40} \zeta(10)
\end{align*}
using $\zeta(2) = \frac{\pi^2}{6}$, $\zeta(4) = \frac{\pi^4}{90}$, $\zeta(6) = \frac{\pi^6}{945}$, $\zeta(8) = \frac{\pi^8}{9450}$, and $\zeta(10) = \frac{\pi^{10}}{93555}$.
Next we rewrite
\[
	4 \zeta(2, 2, 2, 4) + 4 \zeta(2, 2, 4, 2) + 4 \zeta(2, 4, 2, 2).
\]
For this we use
\begin{align*}
	\zeta(2, 2, 2) \zeta(4) - \zeta(2, 2, 2, 4) - {} & \zeta(2, 2, 4, 2) - \zeta(2, 4, 2, 2) - \zeta(4, 2, 2, 2) \\
	&= \zeta(2, 2, 6) + \zeta(2, 6, 2) + \zeta(6, 2, 2) \\
	&= \zeta(2, 2) \zeta(6) - (\zeta(8, 2) + \zeta(2, 8)) \\
	&= \zeta(2, 2) \zeta(6) - (\zeta(2) \zeta(8) - \zeta(10)).
\end{align*}
Therefore $4 \zeta(2, 2, 2, 4) + 4 \zeta(2, 2, 4, 2) + 4 \zeta(2, 4, 2, 2)$ can be written using $\zeta(2, 2) \zeta(6)$, $\zeta(2, 2, 2) \zeta(4)$, $\zeta(4, 2, 2, 2)$, and $\zeta(10)$.
Finally, we use
\[
	\zeta(\underbrace{2,\dots,2}_{j})= \frac{\pi^{2j}}{(2j+1)!}
\]
(see for example \cite{Hoffman}) to write $\zeta(2, 2)$, $\zeta(2, 2, 2)$, and $\zeta(2, 2, 2, 2, 2)$.
Consolidating these results, we obtain
\[
	a_{10} = \tfrac{7}{8} \zeta(10) + 3 \zeta(2, 4, 4) - 12 \zeta(4, 2, 2, 2).
\]
\end{example}

\section{Lucas congruences modulo $p^2$}\label{congruences}

Gessel~\cite{Gessel} proved three theorems on congruences for $A(n)$ where $n \geq 0$.
In this section we generalize these theorems to $n \in \Z$, making substantial use of the reflection formula $A(-1 - z) = A(z)$ from Proposition~\ref{reflection formula}.
We simplify one of the arguments by using the fact that we can differentiate $A(z)$.
We then use these congruences to prove Theorem~\ref{Lucas congruence p^2}.

First we generalize Gessel's result that the Ap\'ery numbers satisfy a Lucas congruence modulo $p$~\cite[Theorem~1]{Gessel}.

\begin{theorem}\label{general Gessel mod p}
Let $p$ be a prime.
For all $d \in \{0, 1, \dots, p - 1\}$ and for all $n \in \Z$, we have $A(d + p n) \equiv A(d) A(n) \mod p$.
\end{theorem}

\begin{proof}
Gessel proved the statement for $n \geq 0$.
Let $n \leq -1$.
By Proposition~\ref{reflection formula},
\begin{align*}
	A(d + p n)
	&= A(-1 - (d + p n)) \\
	&= A((p - 1 - d) + p \, (-1 - n)) \\
	&\equiv A(p - 1 - d) A(-1 - n) \mod p \\
	&= A(p - 1 - d) A(n).
\end{align*}
Malik and Straub~\cite[Lemma~6.2]{Malik--Straub} proved that $A(p - 1 - d) \equiv A(d) \mod p$, which completes the proof.
\end{proof}

Next we generalize Gessel's congruence for $A(p n)$ modulo $p^3$ for $p \geq 5$ and variants for $p = 2$ and $p = 3$~\cite[Theorem~3]{Gessel}.

\begin{theorem}\label{general Gessel mod p^3}
For all $n \in \Z$,
\begin{itemize}
\item
$A(n) \equiv 5^n \mod 8$ for all $n \geq 0$ and $A(n) \equiv 5^{n + 1} \mod 8$ for all $n \leq -1$,
\item
$A(d + 3 n) \equiv A(d) A(n) \mod 9$ for all $d \in \{0, 1, 2\}$, and
\item
$A(p n) \equiv A(n) \equiv A(p n + p - 1) \mod p^3$ for all primes $p \geq 5$.
\end{itemize}
\end{theorem}

A special case of a theorem of Straub~\cite[Theorem~1.2]{Straub} shows that $A(p n) \equiv A(n) \mod p^3$ for all $n \in \Z$ and all primes $p \geq 5$.
We prove this result another way, using an approach similar to Gessel's.

\begin{proof}[Proof of Theorem~\ref{general Gessel mod p^3}]
Gessel proved $A(n) \equiv 5^n \mod 8$ for all $n \geq 0$.
For $n \leq -1$, we use Proposition~\ref{reflection formula} to write
\begin{align*}
	A(n)
	= A(-1 - n)
	&\equiv 5^{-1 - n} \mod 8 \\
	&\equiv 5^{1 + n} \mod 8
\end{align*}
since $5^{-1} \equiv 5 \mod 8$.

For $p = 3$, the proof is similar to the proof of Theorem~\ref{general Gessel mod p}.
Gessel proved the statement for $n \geq 0$, so for $n \leq -1$ we have
\begin{align*}
	A(d + 3 n)
	&= A(-1 - (d + 3 n)) \\
	&= A((2 - d) + 3 (-1 - n)) \\
	&\equiv A(2 - d) A(-1 - n) \mod 9 \\
	&\equiv A(d) A(n) \mod 9
\end{align*}
since one checks that $A(2 - d) \equiv A(d) \mod 9$.

Let $p \geq 5$.
Gessel proved $A(p n) \equiv A(n) \mod p^3$ for all $n \geq 0$.
We show $A(p n + p - 1) \equiv A(n) \mod p^3$ for all $n \geq 0$.
We write
\begin{align*}
	A(p n + p - 1)
	&= \sum_{k = 0}^{p n + p - 1} \binom{p n + p - 1}{k}^2 \binom{p n + p - 1 + k}{k}^2 \\
	&= \sum_{d = 0}^{p - 1} \sum_{m = 0}^n \binom{p n + p - 1}{p m + d}^2 \binom{p \, (n + m + 1) + d - 1}{p m + d}^2 \\
	&= \sum_{d = 0}^{p - 1} \sum_{m = 0}^n \binom{p n + p - 1}{p m + d}^2 \frac{p^2 (n + 1)^2}{(p \, (n + m + 1) + d)^2} \binom{p \, (n + m + 1) + d}{p m + d}^2 \\
	&= S_0 + S_1
\end{align*}
where
\[
	S_0 = \sum_{m = 0}^n \binom{p n + p - 1}{p m}^2 \frac{(n + 1)^2}{(n + m + 1)^2} \binom{p \, (n + m + 1)}{p m}^2
\]
is the summand for $d = 0$, and
\[
	S_1 = \sum_{d = 1}^{p - 1} \sum_{m = 0}^n \binom{p n + p - 1}{p m + d}^2 \frac{p^2 (n + 1)^2}{(p \, (n + m + 1) + d)^2} \binom{p \, (n + m + 1) + d}{p m + d}^2.
\]
For $S_0$, we have
\begin{align*}
	S_0
	&= \sum_{m = 0}^n \frac{(p n + p - p m)^2}{(p n + p)^2} \binom{p n + p}{p m}^2 \frac{(n + 1)^2}{(n + m + 1)^2} \binom{p \, (n + m + 1)}{p m}^2 \\
	&\equiv \sum_{m = 0}^n \frac{(n - m + 1)^2}{(n + m + 1)^2} \binom{n + 1}{m}^2 \binom{n + m + 1}{m}^2 \mod p^3 \\
	&= \sum_{m = 0}^n \binom{n}{m}^2 \binom{n + m}{m}^2 \\
	&= A(n)
\end{align*}
by Jacobsthal's congruence $\binom{p a}{p b} \equiv \binom{a}{b} \mod p^3$, which holds for all primes $p \geq 5$~\cite{Ljunggren--Jacobsthal}.
For $S_1$, we have
\begin{align*}
	S_1
	&\equiv p^2 \sum_{d = 1}^{p - 1} \sum_{m = 0}^n \binom{p n + p - 1}{p m + d}^2 \frac{(n + 1)^2}{d^2} \binom{p \, (n + m + 1) + d}{p m + d}^2 \mod p^3 \\
	&\equiv p^2 \sum_{d = 1}^{p - 1} \sum_{m = 0}^n \binom{p - 1}{d}^2 \binom{n}{m}^2 \frac{(n + 1)^2}{d^2} \binom{d}{d}^2 \binom{n + m + 1}{m}^2 \mod p^3
\end{align*}
by the Lucas congruence for binomial coefficients modulo $p$.
Since
\[
	\binom{p - 1}{d} = \frac{(p - 1) (p - 2) \cdots (p - d)}{1 \cdot 2 \cdots d} \equiv \frac{(-1) (-2) \cdots (-d)}{1 \cdot 2 \cdots d} \equiv (-1)^d \mod p,
\]
we obtain
\begin{align*}
	S_1
	&\equiv p^2 \left(\sum_{d = 1}^{p - 1} \frac{1}{d^2}\right) \sum_{m = 0}^n \binom{n}{m}^2 (n + 1)^2 \binom{n + m + 1}{m}^2 \mod p^3 \\
	&\equiv 0 \mod p^3
\end{align*}
since $\sum_{d = 1}^{p - 1} \frac{1}{d^2} \equiv 0 \mod p$, as established by Wolstenholme~\cite{Wolstenholme}.
Therefore $A(p n + p - 1) = S_0 + S_1 \equiv A(n) \mod p^3$.

Now for $n \leq -1$ we have
\begin{align*}
	A(p n)
	&= A(-1 - p n) \\
	&= A((p - 1) + p \, (-1 - n)) \\
	&\equiv A(-1 - n) \mod p^3 \\
	&= A(n)
\end{align*}
and
\begin{align*}
	A(p n + p - 1)
	&= A(-1 - (p n + p - 1)) \\
	&= A(p \, (-1 - n)) \\
	&\equiv A(-1 - n) \mod p^3 \\
	&= A(n).
	\qedhere
\end{align*}
\end{proof}

Finally, we generalize Gessel's congruence for $A(d + p n)$ modulo $p^2$~\cite[Theorem~4]{Gessel}.
Recall that $A'(n)$ is given by Equation~\eqref{Gessel Apery derivative}.
Since $A'(n) \in \Q$ for every $n \geq 0$, it follows that if the denominator of $A'(n)$ is not divisible by $p$ then we can interpret $A'(n)$ modulo $p^2$.

\begin{theorem}\label{general Gessel mod p^2}
Let $p$ be a prime, and let $d \in \{0, 1, \dots, p - 1\}$.
The denominator of $A'(d)$ is not divisible by $p$.
Moreover, for all $n \in \Z$,
\begin{equation}\label{congruence mod p^2}
	A(d + p n) \equiv \left(A(d) + p n A'(d)\right) A(n) \mod p^2.
\end{equation}
\end{theorem}

\begin{proof}
Gessel proved the statement for $n \geq 0$.
The same approach allows us to prove the general case.

Fix $n \in \Z$.
For each $d \in \{0, 1, \dots, p - 1\}$, define $c_d \in \{0, 1, \dots, p - 1\}$ such that $A(d + p n) \equiv A(d) A(n) + p c_d \mod p^2$; this can be done by Theorem~\ref{general Gessel mod p}.
Let $c_{-1} = 0$.
(The value of $c_{-1}$ does not actually matter, since it will be multiplied by $0$.)
We show that $(c_d)_{0 \leq d \leq p - 1}$ and $(n A'(d) A(n))_{0 \leq d \leq p - 1}$ satisfy the same recurrence and initial conditions modulo $p$; this will imply $c_d \equiv n A'(d) A(n) \mod p$.
Theorem~\ref{general Gessel mod p^3} implies that $A(p n) \equiv A(n) \mod p^2$, so $c_0 = 0$.
Since $A'(0) = 0$, the initial conditions are equal.

Let $d \in \{1, 2, \dots, p - 1\}$.
Write Equation~\eqref{Apery recurrence} as
\begin{equation}\label{abbreviated Apery recurrence}
	\sum_{i = 0}^2 r_i(n) A(n - i) = 0,
\end{equation}
where each $r_i(n)$ is a polynomial in $n$ with integer coefficients.
Note that Equation~\eqref{abbreviated Apery recurrence} holds for all $n \in \Z$.
Substituting $d + p n$ for $n$ in Equation~\eqref{abbreviated Apery recurrence} gives
\[
	\sum_{i = 0}^2 r_i(d + p n) A(d - i + p n) = 0.
\]
If $d - i = -1$ then $r_i(d + p n) = r_2(1 + p n) = (p n)^3 \equiv 0 \mod p^2$, hence the arbitrary value of $c_{-1}$.
Therefore, using the Taylor expansion of $r_i(n)$, we have
\[
	\sum_{i = 0}^2 \big(r_i(d) + p n r_i'(d)\big) \big(A(d - i) A(n) + p c_{d - i}\big) \equiv 0 \mod p^2.
\]
Since $\sum_{i = 0}^2 r_i(d) A(d - i) = 0$, expanding and dividing by $p$ gives
\[
	\sum_{i = 0}^2 \big(r_i(d) c_{d - i} + n r_i'(d) A(d - i) A(n)\big) \equiv 0 \mod p.
\]
This gives a recurrence satisfied by $(c_d)_{0 \leq d \leq p - 1}$ that can be used to compute $c_1, c_2, \dots, c_{p - 1}$ since $r_0(d) = d^3 \nequiv 0 \mod p$.

To obtain a recurrence for $(n A'(d) A(n))_{0 \leq d \leq p - 1}$, we differentiate Equation~\eqref{interpolation} to obtain
\[
	\sum_{i = 0}^2 \big(r_i(d) A'(d - i) + r_i'(d) A(d - i)\big) = 0.
\]
Since $A'(0)$ and $A'(1)$ are integers and $r_0(d) \nequiv 0 \mod p$, the denominator of $A'(d)$ is not divisible by $p$.
By multiplying by $n A(n)$, we obtain
\[
	\sum_{i = 0}^2 \big(r_i(d) n A'(d - i) A(n) + n r_i'(d) A(d - i) A(n)\big) = 0.
\]
By subtracting this from the recurrence for $(c_d)_{0 \leq d \leq p - 1}$, we see that
\[
	\sum_{i = 0}^2 r_i(d) \big(c_{d - i} - n A'(d - i) A(n)\big) \equiv 0 \mod p.
\]
Since $r_0(d) \nequiv 0 \mod p$, it follows that $c_d \equiv n A'(d) A(n) \mod p$ for all $d \in \{0, 1, \dots, p - 1\}$.
\end{proof}

In the case $p = 3$, Theorem~\ref{general Gessel mod p^2} gives a second proof of the congruence $A(d + 3 n) \equiv A(d) A(n) \mod 9$ from Theorem~\ref{general Gessel mod p^3}, since $A'(0) \equiv A'(1) \equiv A'(2) \equiv 0 \mod 3$.
For larger primes, in general $A(d + p n) \nequiv A(d) A(n) \mod p^2$.
However, if we restrict to certain sets of base-$p$ digits, then we do obtain congruences that hold modulo $p^2$.
For example, if $d \in \{0, 2, 4\}$, then
\[
	A(d + 5 n) \equiv A(d) A(n) \mod 25.
\]
This was proven by the authors~\cite{Rowland--Yassawi diagonals} by computing an automaton for $A(n) \bmod 25$.
Since $A(0) \equiv 1 \equiv A(4) \mod 25$ and $A(2) \equiv 23 \mod 25$, this implies $A(n) \equiv 23^{e_2(n)} \mod 25$ for all $n \geq 0$ whose base-$5$ digits belong to $\{0, 2, 4\}$, where $e_2(n)$ is the number of $2$s in the base-$5$ representation of $n$.
Theorem~\ref{Lucas congruence p^2}, reformulated as the following theorem, generalizes this result to other primes.

We say that the set $D \subseteq \{0, 1, \dots, p - 1\}$ supports a \emph{Lucas congruence} for the sequence $s(n)_{n \in \Z}$ modulo $p^\alpha$ if $s(d + p n) \equiv s(d) s(n) \mod p^\alpha$ for all $d \in D$ and for all $n \in \Z$.
As mentioned in the proof of Theorem~\ref{general Gessel mod p}, Malik and Straub~\cite[Lemma~6.2]{Malik--Straub} proved that $A(d) \equiv A(p - 1 - d) \mod p$ for each $d \in \{0, 1, \dots, p - 1\}$.
Let $D(p)$ be the set of base-$p$ digits for which this congruence holds modulo $p^2$; that is,
\[
	D(p) = \left\{d \in \{0, 1, \dots, p - 1\} : A(d) \equiv A(p - 1 - d) \mod p^2\right\}.
\]

\begin{theorem}
The set $D(p)$ is the maximum set of digits that supports a Lucas congruence for the Ap\'ery numbers modulo $p^2$.
\end{theorem}

\begin{proof}
Let $d \in D(p)$, so that $A(d) \equiv A(p - 1 - d) \mod p^2$.
Letting $n = -1$ in Theorem~\ref{general Gessel mod p^2} gives $A(d - p) \equiv A(d) - p A'(d) \mod p^2$.
Applying Proposition~\ref{reflection formula}, we find
\begin{align*}
	p A'(d)
	&\equiv A(d) - A(d - p) \mod p^2 \\
	&= A(d) - A(p - 1 - d) \\
	&\equiv 0 \mod p^2.
\end{align*}
Therefore it follows from Theorem~\ref{general Gessel mod p^2} that, for all $n \in \Z$,
\begin{align*}
	A(d + p n)
	&\equiv \left(A(d) + p n A'(d)\right) A(n) \mod p^2 \\
	&\equiv A(d) A(n) \mod p^2.
\end{align*}
Therefore $D(p)$ supports a Lucas congruence for the Ap\'ery numbers modulo $p^2$.

To see that $D(p)$ is the maximum such set, assume $A(d + p n) \equiv A(d) A(n) \mod p^2$ for all $n \in \Z$.
Then
\begin{align*}
	\left(A(d) + p n A'(d)\right) A(n)
	&\equiv A(d + p n) \mod p^2 \\
	&\equiv A(d) A(n) \mod p^2,
\end{align*}
and it follows that $p n A'(d) A(n) \equiv 0 \mod p^2$ for all $n \in \Z$.
Therefore $A(d) - A(p - 1 - d) = A(d) - A(d - p) \equiv p A'(d) \equiv 0 \mod p^2$.
\end{proof}

As a special case, we obtain the following congruence, since $\{0, p - 1\} \subseteq D(p)$ by Theorem~\ref{general Gessel mod p^3}, 
and $A(0) = 1 \equiv A(p - 1) \mod p^2$.

\begin{corollary}\label{Lucas congruence p^2 - special case}
Let $p \neq 2$ and $n \geq 0$.
If the base-$p$ digits of $n$ all belong to $\{0, \frac{p - 1}{2}, p - 1\}$, then $A(n) \equiv A(\frac{p - 1}{2})^{e(n)} \mod p^2$ where $e(n)$ is the number of occurrences of the digit $\frac{p - 1}{2}$.
\end{corollary}

These are the first several primes with digit sets $D(p)$ containing at least $4$ digits:
\[
	\begin{array}{c|c}
		p	& D(p) \\ \hline
		7	& \{0, 2, 3, 4, 6\} \\
		23	& \{0, 7, 11, 15, 22\} \\
		43	& \{0, 5, 18, 21, 24, 37, 42\} \\
		59	& \{0, 6, 29, 52, 58\} \\
		79	& \{0, 18, 39, 60, 78\} \\
		103	& \{0, 17, 51, 85, 102\} \\
		107	& \{0, 14, 21, 47, 53, 59, 85, 92, 106\} \\
		127	& \{0, 17, 63, 109, 126\} \\
		131	& \{0, 62, 65, 68, 130\} \\
		139	& \{0, 68, 69, 70, 138\} \\
		151	& \{0, 19, 75, 131, 150\} \\
		167	& \{0, 35, 64, 83, 102, 131, 166\}
	\end{array}
\]

A natural question, which we do not address here, is the following.
How big can $\lvert D(p) \rvert$ be, as a function of $p$?

Theorem~\ref{general Gessel mod p^3} implies the following Lucas congruence modulo $p^3$.

\begin{theorem}\label{Lucas congruence p^3}
Let $p \geq 5$ and $n \geq 0$.
If the base-$p$ digits of $n$ all belong to $\{0, p - 1\}$, then $A(n) \equiv 1 \mod p^3$.
\end{theorem}

Experiments do not suggest the existence of any additional Lucas congruences for the Ap\'ery numbers modulo $p^3$.
We leave this as open question.

\section*{Acknowledgement}

We thank Manon Stipulanti for productive discussions.

\end{document}